\documentclass[smallextended]{svjour3}
\usepackage{graphicx}
\usepackage{mathptmx}
\usepackage{amsmath,amsfonts,color}
\newcommand{\R}{\mathbb{R}}
\newcommand{\fInt}{\mathcal{I}}
\newcommand{\iprod}[1]{\langle#1\rangle}
\newcommand{\bigiprod}[1]{\bigl\langle#1\bigr\rangle}
\newcommand{\Mult}{\mathcal{M}}

\newcommand{\Sh}{\mathbb{S}_h}
\newcommand{\Uref}{U_{\text{ref}}}
\begin{document}
\title{{ A semidiscrete} finite element approximation of a time-fractional
Fokker--Planck equation, non-smooth initial data\thanks{This work
was supported by the Australian Research Council grant DP140101193.}}
\titlerunning{Time-fractional Fokker--Planck equation}
\author{Kim Ngan Le \and William McLean \and Kassem Mustapha}
\institute{
Kim Ngan Le \at
School of Mathematics and Statistics,
The University of New South Wales, Sydney 2052, Australia.\\
%Tel.: +612-9385-7111\\
%Fax.: +612-9385-7123\\
\email{n.le-kim@unsw.edu.au}
\and
William McLean \at
School of Mathematics and Statistics,
The University of New South Wales, Sydney 2052, Australia.\\
%Tel.: +612-9385-7111\\
%Fax.: +612-9385-7123\\
\email{w.mclean@unsw.edu.au}
\and
Kassem Mustapha \at
Department of Mathematics and Statistics,
King Fahd University of Petroleum and Minerals,
Dhahran 31261, Saudi Arabia.\\
\email{kassem@kfupm.edu.sa}}

\date{\today}

\maketitle

\begin{abstract}
We present a new stability and convergence analysis for the spatial
discretisation of a time-fractional Fokker--Planck equation in a
polyhedral domain, using continuous, piecewise-linear, finite
elements. The forcing may depend on time as well as on the spatial
variables, and the initial data may have low regularity.
Our analysis uses a novel sequence of energy arguments in combination
with a generalised Gronwall inequality.  Although this theory covers
only the spatial discretisation, we present numerical experiments with
a fully-discrete scheme employing a very small time step, and observe
results consistent with the predicted convergence behaviour.
 % Integral equations > Integro-PDEs
\end{abstract}
%\documentclass[a4paper,12pt]{article}
%\usepackage[utf8]{inputenc}
%\usepackage{amsmath,amsfonts,amsthm,graphicx}
%\usepackage{showlabels}
%\usepackage{hyperref}
%%%%%%%%%%%%%%%%%%%%%%%%%%%%%%%%%%%%%%%%%%%%%%%%%%%%%%%%%%%%%%%%%%%%%%
\section{Introduction}
We consider the spatial discretisation via Galerkin finite elements
of a time-fractional Fokker--Planck
equation~\cite{AngstmannEtAl2015,HenryLanglandsStraka2010},
\begin{equation}\label{eq: FFP}
\partial_t u
	-\nabla\cdot\bigl(
	\partial_t^{1-\alpha}\kappa_\alpha\nabla u
	-\vec{F}\partial_t^{1-\alpha}u\bigr)={ 0}
	\quad\text{for $\vec{x}\in\Omega$ and $0<t<T$,}
\end{equation}
{ with  initial condition $u(\vec{x},0)=u_0(\vec{x}),$} where $\partial_t=\partial/\partial t$ and $\Omega$ is a polyhedral
domain in~$\R^d$ ($d\ge1$). The fractional exponent is restricted
to the range~$0<\alpha<1$, { $\kappa_\alpha>0$ is the diffusivity coefficient}. In our analysis, { we put
$\kappa_\alpha=1$ for convenience, but it is
straight forward to extend our methods to allow for a
spatially-varying diffusivity.} The fractional
derivative is taken in the Riemann--Liouville sense, that is,
$\partial_t^{1-\alpha}u=\partial_t\fInt^\alpha u$, where the
fractional integration operator~$\fInt^\alpha$ is defined by
\[
\fInt^\alpha u(t)=\omega_\alpha*u(t)
	=\int_0^t\omega_\alpha(t-s)u(s)\,ds,
	\quad\omega_\alpha(t)=\frac{t^{\alpha-1}}{\Gamma(\alpha)}.
\]
{ Though we impose  a homogeneous Dirichlet boundary condition,
\begin{equation}\label{eq: Dirichlet bc}
u(\vec{x},t)=0\quad\text{for $\vec{x}\in\partial\Omega$ and $0<t<T$,}
\end{equation}
 the proposed stability and errors analysis remain valid for zero-flux boundary condition, see Remark \ref{rem: zro flux}.}

{ The time-space dependent driving force~$\vec{F}$ and it time partial derivative, $\partial_t\vec{F}$,  are assumed to be in ~$L_\infty\bigl(\Omega\times(0,T),\R^d\bigr)$. When $\vec{F}$} is independent of~$t$, the model
problem \eqref{eq: FFP} can be rewritten in the form
\begin{equation}\label{eq: F(x)}
\fInt^{1-\alpha}(\partial_t u)-\nabla\cdot\bigl(\kappa_\alpha
	\nabla u-\vec{F}(\vec{x}) u \bigr)={ 0},
\end{equation}
where the first term is just the Caputo fractional derivative of
order~$\alpha$. For a one- or
two-dimensional spatial domain~$\Omega$, numerical methods applicable
to~\eqref{eq: F(x)} have been widely
studied~\cite{CaoFuHuang2012,ChenLiuZhuangAnh2009,Cui2015,Deng2007,%
FairweatherZhangYangXu2015,GaoSun2015a,GaoSun2015b,%
GraciaORiordanStynes2015,Jiang2015,SaadatmandiDehghanAzizi2012,%
VongWang2015,WeiZhangHe2013,ZhengLiuTurnerAnh2015}. In all of these
works, the solution~$u$ was assumed to be sufficiently regular,
including at~$t=0$.  Although \eqref{eq: F(x)} is in many
respects more convenient for constructing and analyzing the accuracy
of numerical schemes, only \eqref{eq: FFP} is physically
valid for a time-dependent
forcing~$\vec{F}$~\cite{HeinsaluPatriarcaGoychukHanggi2007}.

Our earlier paper~\cite{KimMcLeanMustapha2016} presented an analysis
of the semidiscrete finite element solution of~\eqref{eq: FFP}
that is limited to cases in which
\begin{enumerate}
\item the solution~$u$ is sufficiently regular,
\item the spatial domain~$\Omega$ is an interval on the real line
(that is, $d=1$),
\item the fractional exponent is in the range $1/2<\alpha<1$,
\item the boundary condition is of homogeneous Dirichlet
type~\eqref{eq: Dirichlet bc}.
\end{enumerate}
By employing { a different approach that based on}  novel energy arguments, 
%in combination with the generalised Gronwall inequality of Dixon and
%McKee~\cite{DixonMcKee1986}, 
we are able to relax significantly the
regularity requirements on~$u$, in addition to permitting $d\ge1$,
$0<\alpha<1$, and zero-flux~\eqref{eq: Neuman bc} as well as Dirichlet
boundary conditions.  This new approach leads to an error bound of
optimal order in~$L_2(\Omega)$ at each fixed~$t>0$, even for
non-smooth initial data~$u_0$. We consider only
continuous piecewise linear elements and (unlike our earlier
paper~\cite{KimMcLeanMustapha2016}) do not analyse any time
discretisation.

In Section~\ref{sec: FEM}, we define the semidiscrete finite
element scheme and outline our main results in the context of our
previous work~\cite{KimMcLeanMustapha2016}.
Section~\ref{sec: fractional} gathers together some
technical estimates involving fractional integrals.
Section~\ref{sec: stability} presents the new
stability result (Theorem~\ref{thm: stability}) and
Section~\ref{sec: error} the new error bound
(Theorem~\ref{thm: error}).  Finally, in Section~\ref{sec: numerical},
we discuss two numerical examples.  The first confirms both the
convergence rate and the dependence on~$t$ predicted by our theory.
The second looks briefly at how the method behaves when~$u_0$ is a
point mass, and therefore does not even belong to~$L_2(\Omega)$.
%%%%%%%%%%%%%%%%%%%%%%%%%%%%%%%%%%%%%%%%%%%%%%%%%%%%%%%%%%%%%%%%%%%%%%
\section{The finite element solution}\label{sec: FEM}
%We henceforth assume for simplicity that $\kappa_\alpha=1$~and
%$g\equiv0$, so that \eqref{eq: FFP gen} reduces to
%\begin{equation}\label{eq: FFP}
%\partial_t u
%	-\nabla\cdot\bigl(\partial_t^{1-\alpha}\nabla u
%	-\vec{F}\partial_t^{1-\alpha}u\bigr)=0
%	\quad\text{for $\vec{x}\in\Omega$ and $0<t<T$.}
%\end{equation}
{ The continuous solution  $u:(0,T]\to H^1_0(\Omega)$ of problem \eqref{eq: FFP} subject to the homogeneous Dirichlet boundary condition \eqref{eq: Dirichlet bc}, satisfies the weak form},
\begin{equation}\label{eq: FFP weak}
\iprod{\partial_tu,v}+\iprod{\partial_t^{1-\alpha}\nabla u,\nabla v}
	-\iprod{\vec{F}\partial_t^{1-\alpha}u,\nabla v}=0
\end{equation}
for all $v\in H^1_0(\Omega)$, where $\iprod{u,v}=\int_\Omega uv$~and
$\iprod{\vec{u},\vec{v}}=\int_\Omega\vec{u}\cdot\vec{v}$.  Let $h$
denote the maximum element diameter from a shape-regular triangulation
of~$\Omega$, and let $\Sh\subseteq H^1_0(\Omega)$ denote the usual
space of continuous, piecewise-linear functions that vanish
on~$\partial\Omega$.  The semidiscrete finite element
solution~$u_h:[0,T]\to\Sh$ is then defined by
\begin{equation}\label{eq: FEM}
\iprod{\partial_tu_h,\chi}
	+\iprod{\partial_t^{1-\alpha}\nabla u_h,\nabla\chi}
	-\iprod{\vec{F}\partial_t^{1-\alpha}u_h,\nabla\chi}=0
	\quad\text{for all $\chi\in\Sh$,}
\end{equation}
together with the initial condition~$u_h(0)=u_{0h}$, where
$u_{0h}\in\Sh$ is a suitable approximation to~$u_0$.

Previously, { for~$0\le t\le T$}, we
showed~\cite[Theorems 3.3~and 3.4]{KimMcLeanMustapha2016} that, $\|u_h(t)\|\le C\|u_{0h}\|_1$
and, provided $u_{0h}$ is chosen to be the Ritz projection of~$u_0$
onto~$\Sh$,
\begin{equation}\label{eq: prev error}
\|u_h(t)-u(t)\|\le
Ch^2\biggl(\|u_0\|_2^2+\int_0^t\|u'(s)\|_2^2\,ds
	\biggr)^{1/2}.
\end{equation}
Here, $\|v\|=\sqrt{\iprod{v,v}}$ denotes the norm in~$L_2(\Omega)$,
$u'(t)=\partial_tu$,
\[
\|v\|_r=\|(-\nabla^2)^{r/2}v\|=\biggl(\sum_{m=1}^\infty\lambda_m^r
	\iprod{v,\varphi_m}^2\biggr)^{1/2}\quad\text{for $r\ge0$,}
\]
and $\varphi_1$, $\varphi_2$, $\varphi_3$, \dots is a complete
orthonormal system in~$L_2(\Omega)$ consisting of
Dirichlet eigenfunctions of the Laplacian:
$\iprod{\varphi_m,\varphi_k}=\delta_{mk}$ and
\[
-\nabla^2\varphi_m=\lambda_m\varphi_m\quad\text{in $\Omega$,}
\quad\text{with $\varphi_m=0$ on~$\partial\Omega$.}
\]
The associated function
space~$\dot H^r(\Omega)=\{\,v\in L_2(\Omega):\|v\|_r<\infty\,\}$ is a
subspace of the usual Sobolev space~$H^r(\Omega)$
for~$0\le r\le1$; in particular, $\dot H^0(\Omega)=L_2(\Omega)$~and
$\dot H^1(\Omega)=H^1_0(\Omega)$.
Also, $\dot H^2(\Omega)=H^2(\Omega)\cap H^1_0(\Omega)$ provided
$\Omega$ is convex (so the Poisson problem is $H^2$-regular).

We prove in Theorem~\ref{thm: stability} a stronger stability
estimate,
\begin{equation}\label{eq: stability}
\|u_h(t)\|\le C\|u_{0h}\|\quad\text{for $0\le t\le T$.}
\end{equation}
Also, whereas the previous error bound~\eqref{eq: prev error} is
meaningful only if $u_0\in\dot H^2(\Omega)$ and
$u'\in L_2\bigl((0,T),\dot H^2(\Omega)\bigr)$, our new error analysis
makes a much weaker regularity assumption: for some~$r$ in the
range~$0\le r\le2$ there is a constant~$K_r$ such that
\begin{equation}\label{eq: u reg}
\|u(t)\|_2+t\|u'(t)\|_2\le t^{\alpha(r-2)/2}K_r
	\quad\text{for $0<t\le T$.}
\end{equation}
When $\vec{F}\equiv\vec{0}$ and the domain~$\Omega$ is convex,
it is known \cite[Theorem~4.4]{McLean2010} that such an estimate
holds with $K_r=C\|u_0\|_r$ in the case of Dirichlet boundary
conditions~\eqref{eq: Dirichlet bc}.  Since the term
of~\eqref{eq: FFP} involving $\vec{F}$ is of lower order in the spatial
variables, we conjecture that the same is true for a nonzero (but
sufficiently regular) forcing~$\vec{F}$.  In Theorem~\ref{thm: error},
we show that if $u_{0h}$ is chosen to be the $L_2$-projection of~$u_0$
onto~$\Sh$, then
\begin{equation}\label{eq: error}
\|u_h(t)-u(t)\|\le Ct^{\alpha(r-2)/2}h^2K_r
\quad\text{for $0\le t\le T$~and $0\le r\le2$.}
\end{equation}
For instance, in the worst case when~$r=0$, the error is
$O(t^{-\alpha}h^2)$.

\begin{remark}\label{rem: zro flux}
{ If we impose  a zero-flux boundary condition,
\begin{equation}\label{eq: Neuman bc}
\partial_t^{1-\alpha}\kappa_\alpha\frac{\partial u}{\partial n}
	-(\vec{F}\cdot\vec{n})\,\partial_t^{1-\alpha}u=0
	\quad\text{for $\vec{x}\in\partial\Omega$ and $0<t<T$,}
\end{equation}
where $\vec{n}$ denotes the outward unit normal to~$\Omega$,} then $u:(0,T]\to H^1(\Omega)$ satisfies \eqref{eq: FFP weak} for all
$v\in H^1(\Omega)$.  Likewise, $u_h$ is defined as in~\eqref{eq: FEM}
but the finite element space~$\Sh\subseteq H^1(\Omega)$ now consists
of \emph{all} continuous piecewise-linear functions (that is, the
elements of~$\Sh$ need not vanish on~$\partial\Omega$).  The
stability estimate~\eqref{eq: stability} remains valid, and the error
bound~\eqref{eq: error} holds assuming $u$
satisfies~\eqref{eq: u reg}, where $\|\cdot\|_2$ is now the norm
in~$H^2(\Omega)$ rather than~$\dot H^2(\Omega)$.  Note that for
either choice of boundary condition, the variational
equation~\eqref{eq: FEM} is equivalent to a system of Volterra
integral equations~\cite[Theorem~3.1]{KimMcLeanMustapha2016} that
admits a unique continuous solution~$u_h:[0,T]\to\Sh$.  Moreover,
the methods of Miller and
Feldstein~\cite[Theorem~1]{MillerFeldstein1971} show that $u_h$ is
continuously differentiable on~$(0,T]$. Finally, notice that in the
case of the zero-flux boundary condition~\eqref{eq: Neuman bc}, the total mass~$\int_\Omega u(\cdot,t)$
within~$\Omega$ is conserved.\end{remark}
%%%%%%%%%%%%%%%%%%%%%%%%%%%%%%%%%%%%%%%%%%%%%%%%%%%%%%%%%%%%%%%%%%%%%%
\section{Fractional integrals}\label{sec: fractional}
In this section only, $C$ is an absolute constant.  Our analysis of
the semidiscrete finite element solution~$u_h$ will rely on the following
technical lemmas, in which $\phi$~and $\psi$ are suitably regular
functions of~$t>0$ taking values in a Hilbert space.

\begin{lemma}\label{lem: I nu mu}
If~ $0\le\mu\le\nu\le1$, then
\[
\int_0^t\|\fInt^\nu\phi\|^2\,ds\le Ct^{2(\nu-\mu)}
	\int_0^t \|\fInt^\mu\phi\|^2\,ds.
\]
\end{lemma}
\begin{proof}
If $\mu=\nu$ then there is nothing to prove, so assume $\mu<\nu$.
In a previous paper~\cite[Lemma~2.3]{KimMcLeanMustapha2016}, we showed
that for~$0<\alpha\le1$,
\[
\int_0^T\|\fInt^\alpha\psi(t)\|^2\,dt\le\omega_{\alpha+1}(T)
	\int_0^T\omega_\alpha(T-t)\int_0^t\|\psi(s)\|^2\,ds\,dt,
\]
and the right-hand side is bounded
by~$\omega_{\alpha+1}(T)^2\int_0^T\|\psi(s)\|^2\,ds$.
Putting $\psi=\fInt^\mu\phi$ and $\alpha=\nu-\mu$, it follows that
$\fInt^\alpha\psi=\fInt^\nu\phi$~and
$\omega_{\alpha+1}(T)\le CT^\alpha=CT^{\nu-\mu}$. \qed
\end{proof}

\begin{lemma}
If~ $0<\alpha<1$ and $\epsilon>0$, then
\begin{equation}\label{eq: fractional A}
\biggl|\int_0^t\iprod{\phi,\fInt^\alpha\psi}\,ds\biggr|
	\le\frac{1}{4\epsilon(1-\alpha)^2}
	\int_0^t\iprod{\fInt^\alpha\phi,\phi}\,ds
	+\epsilon\int_0^t\iprod{\fInt^\alpha\psi,\psi}\,ds,
\end{equation}
\begin{equation}\label{eq: fractional D}
\int_0^t\|\fInt^\alpha\phi\|^2\,ds
	\le\frac{Ct^\alpha}{1-\alpha}
		\int_0^t \iprod{\fInt^\alpha\phi,\phi}\,ds,
\end{equation}
\begin{equation}\label{eq: fractional E}
\int_0^t\iprod{\phi,\fInt^\alpha\phi}\,ds
	\le Ct^\alpha\int_0^t\|\phi\|^2\,ds.
\end{equation}
\end{lemma}
\begin{proof}
From a result of Mustapha~and
Sch\"otzau~\cite[Lemma~3.1(iii)]{MustaphaSchoetzau2014},
\[
\biggl|\int_0^t\iprod{\phi,\fInt^\alpha\psi}\,ds\biggr|
	\le\frac{1}{\cos(\alpha\pi/2)}
	\biggl(\int_0^t\iprod{\phi,\fInt^\alpha\phi}\,ds\biggr)^{1/2}
	\biggl(\int_0^t\iprod{\psi,\fInt^\alpha\psi}\,ds\biggr)^{1/2},
\]
so \eqref{eq: fractional A} follows because
$\cos(\alpha\pi/2)\ge1-\alpha$.  The same
paper~\cite[Lemma~3.1(ii)]{MustaphaSchoetzau2014} showed that
\begin{equation}\label{eq: coercivity}
\int_0^t\iprod{\phi,\fInt^\alpha\phi}\,ds
	\ge\cos(\pi\alpha/2)\int_0^t\|\fInt^{\alpha/2}\phi\|^2\,ds,
\end{equation}
and by choosing $\nu=\alpha$~and $\mu=\alpha/2$
in Lemma~\ref{lem: I nu mu} have
\[
\int_0^t\|\fInt^\alpha\phi\|^2\,ds\le Ct^\alpha
	\int_0^t\|\fInt^{\alpha/2}\phi\|^2\,ds,
\]
proving \eqref{eq: fractional D}.  Instead choosing
$\nu=\alpha$~and $\mu=0$ in Lemma~\ref{lem: I nu mu} gives
\[
\int_0^t\|\fInt^\alpha\phi\|^2\,ds
	\le Ct^{2\alpha}\int_0^t\|\phi\|^2\,ds,
\]
so
\[
\int_0^t\iprod{\phi,\fInt^\alpha\phi}\,ds
	\le\biggl(\int_0^t\|\phi\|^2\,ds\biggr)^{1/2}
	\biggl(\int_0^t\|\fInt^\alpha\phi\|^2\,ds\biggr)^{1/2}
	\le Ct^\alpha\int_0^t\|\phi\|^2\,ds,
\]
proving \eqref{eq: fractional E}.\qed
\end{proof}

\begin{lemma}\label{lem: y(t)}
If~ $0<\alpha<1$, then
\[
\int_0^t\|\fInt^\alpha\phi\|^2\le\frac{Ct^{\alpha/2}}{1-\alpha}
\int_0^t \omega_{\alpha/2}(t-s)y(s)\,ds
\quad\text{for}\quad
y(t)=\int_0^t\iprod{\phi,\fInt^\alpha\phi}\,ds.
\]
\end{lemma}
\begin{proof}
From our earlier paper~\cite[Lemma~2.3]{KimMcLeanMustapha2016},
\[
\int_0^T\|\fInt^\nu\psi(t)\|^2\,dt\le\omega_{\nu+1}(T)
	\int_0^T\omega_\nu(T-t)\int_0^t\|\psi(s)\|^2\,ds\,dt,
\]
so the result follows by letting $\nu=\alpha/2$ and
$\psi=\fInt^{\alpha/2}\phi$, and then using \eqref{eq: coercivity}.
\qed
\end{proof}

\begin{lemma}\label{lem: phi(t)-phi(0)}
If~ $0<\alpha<1$, then
\[
\|\phi(t)-\phi(0)\|^2\le\frac{t^{1-\alpha}}{(1-\alpha)^2}\int_0^t
	\iprod{\phi'(s),(\fInt^\alpha\phi')(s)}\,ds.
\]
\end{lemma}
\begin{proof}
We showed previously \cite[Lemma~2.1]{KimMcLeanMustapha2016} that
\[
\|\phi(t)-\phi(0)\|^2\le\frac{t^{1-\alpha}}{1-\alpha}\int_0^t
	\|\fInt^{\alpha/2}\phi'(s)\|^2\,ds,
\]
so the desired estimate follows from~\eqref{eq: coercivity} and the
inequality~$\cos(\alpha\pi/2)\ge1-\alpha$.\qed
\end{proof}

%%%%%%%%%%%%%%%%%%%%%%%%%%%%%%%%%%%%%%%%%%%%%%%%%%%%%%%%%%%%%%%%%%%%%%
\section{Stability}\label{sec: stability}
We seek to estimate the finite element solution~$u_h(t)$ in terms of
the initial data~$u_{0h}$.  Throughout, the generic constant~$C$ may
depend on $\alpha$, $T$ and the { vector norms of $\vec{F}$~and
$\vec{F}'=\partial_t\vec{F}$ in~$L_\infty(\Omega\times(0,T))$}.

It will be convenient to define
\begin{equation}\label{eq: M B}
\begin{aligned}
\Mult \phi(t)&=t\phi(t),&
\vec{B}_1(\phi)&=\fInt^1(\vec{F}\partial_t^{1-\alpha}\phi),\\
\vec{B}_2(\phi)&=(\Mult-\alpha\fInt){ \vec{B}_1(\phi)},&
\vec{B}_3(\phi)&=[\Mult{ \vec{B}_1(\phi)}]',
\end{aligned}
\end{equation}
and we will use the elementary identities
\begin{equation}\label{eq: M I}
\Mult\fInt^\alpha-\fInt^\alpha\Mult=\alpha\fInt^{\alpha+1}
\end{equation}
and
\begin{equation}\label{eq: Caputo}
(\partial_t^{1-\alpha}
\phi)(t)=(\fInt^\alpha\phi)'=\phi(0)\omega_\alpha(t)
	+(\fInt^\alpha\phi')(t).
\end{equation}

\begin{lemma}\label{lem: B integrals}
For $0\le t\le T$,
\begin{gather*}
\int_0^t\|\vec{B}_1(\phi)\|^2\,ds\le C
	\int_0^t\|\fInt^\alpha\phi\|^2\,ds,\qquad
\int_0^t\|\vec{B}_2(\phi)\|^2\,ds\le Ct^2
	\int_0^t\|\fInt^\alpha\phi\|^2\,ds,\\
\int_0^t\|\vec{B}_3(\phi)\|^2\,ds
	\le C\int_0^t\bigl(\|\fInt^\alpha(\Mult\phi)'\|^2
	+\|\fInt^\alpha(\Mult\phi)\|^2+\|\fInt^\alpha\phi\|^2\bigr)\,ds.
\end{gather*}
\end{lemma}
\begin{proof}
Integration by parts { (in time)} shows that
\begin{equation}\label{eq: by parts}
\vec{B}_1(\phi)%=\int_0^t\vec{F}(\fInt^\alpha\phi)'\,ds
=\vec{F}\fInt^\alpha\phi-\fInt^1(\vec{F}'\fInt^\alpha\phi),
\end{equation}
and our assumptions on~$\vec{F}$ imply
\begin{equation}\label{eq: intermediate}
\|\vec{F}\fInt^\alpha\phi\|^2\le C\|\fInt^\alpha\phi\|^2
\quad\text{and}\quad
\|\fInt^1(\vec{F}'\fInt^\alpha\phi)\|^2\le
Ct\int_0^t\|\fInt^\alpha\phi\|^2\,ds,
\end{equation}
so the first estimate follows at once. { The second estimate follows immediately  from the first one and the inequality 
\[
\|\vec{B}_2(\phi)(s)\|^2\le Cs^2\|\vec{B}_1(\phi)(s)\|^2
	+C s\fInt^1(\|\vec{B}_1(\phi)\|^2)(s).\]}
With the help of the identities \eqref{eq: by parts}~and
\eqref{eq: M I}, we find that
\begin{align*}
\Mult\vec{B}_1(\phi)
	&=\Mult\bigl(\vec{F}\fInt^\alpha\phi
		-\fInt^1(\vec{F}'\fInt^\alpha\phi)\bigr)
	=\vec{F}(\fInt^\alpha\Mult\phi+\alpha\fInt^{\alpha+1}\phi)
		-\Mult\fInt^1(\vec{F}'\fInt^\alpha\phi)
\end{align*}
so
\begin{multline*}
\vec{B}_3(\phi)=\vec{F}'(\fInt^\alpha\Mult\phi
	+\alpha\fInt^{\alpha+1}\phi)
	+\vec{F}(\fInt^\alpha(\Mult\phi)'+\alpha\fInt^\alpha\phi)\\
-\fInt^1(\vec{F}'\fInt^\alpha\phi)
		-\Mult\vec{F}'\fInt^\alpha\phi.
\end{multline*}
Thus,
\begin{multline*}
\|\vec{B}_3(\phi)\|^2
	\le C\bigl(\|\fInt^\alpha\Mult\phi\|^2
	+\|\fInt^\alpha(\Mult\phi)'\|^2\bigr)
	+C(1+t^2)\|\fInt^\alpha\phi\|^2\\
	+Ct\int_0^t\|\fInt^\alpha\phi\|^2\,ds,
\end{multline*}
which implies the third estimate.\qed
\end{proof}

In the next two lemmas, we prove preliminary stability estimates for
$u_h$~and $\Mult u_h$.

\begin{lemma}\label{lem: stability A}
The finite element solution satisfies, for $0\le t\le T$,
\[
\int_0^t\bigl(\iprod{u_h,\fInt^\alpha u_h}
	+\|\fInt^\alpha\nabla u_h\|^2\bigr)\,ds
	\le Ct^{1+\alpha}\|u_{0h}\|^2
\]
and
\[
\int_0^t\|\fInt^\alpha u_h\|^2\,ds\le Ct^{1+2\alpha}\|u_{0h}\|^2.
\]
\end{lemma}
\begin{proof}
We integrate \eqref{eq: FEM} in time to obtain
\begin{equation}\label{eq: I FEM}
\iprod{u_h(t),\chi}+\iprod{(\fInt^\alpha\nabla u_h)(t),\nabla\chi}
	-\iprod{\vec{B}_1(u_h)(t),\nabla\chi}
	=\iprod{u_{0h},\chi}
\end{equation}
and then choose~$\chi=\fInt^\alpha u_h(t)$ so that
\begin{align*}
\iprod{u_h,\fInt^\alpha u_h}
	+\|\fInt^\alpha\nabla u_h\|^2&=
\iprod{\vec{B}_1(u_h),\fInt^\alpha\nabla u_h}
	+\iprod{u_{0h},\fInt^\alpha u_h}\\
	&\le\tfrac12\|\vec{B}_1(u_h)\|^2
    +\tfrac12\|\fInt^\alpha\nabla u_h\|^2
	+\iprod{u_{0h},\fInt^\alpha u_h}.
\end{align*}
Therefore, after cancelling the
term~$\tfrac12\|\fInt^\alpha\nabla u_h\|^2$, integrating in time
and applying Lemma~\ref{lem: B integrals}, we deduce that
\begin{equation}\label{eq: no F}
\int_0^t\bigl(
\iprod{u_h,\fInt^\alpha u_h}+\tfrac12\|\fInt^\alpha\nabla u_h\|^2
\bigr)\,ds\le C\int_0^t\|\fInt^\alpha u_h\|^2\,ds
	+\int_0^t\iprod{u_{0h},\fInt^\alpha u_h} \,ds.
\end{equation}
From~\eqref{eq: fractional A} with $\phi=u_{0h}$~and $\psi=u_h$,
\[
\int_0^t\iprod{u_{0h},\fInt^\alpha u_h} \,ds\le
	C\int_0^t\iprod{u_{0h},\fInt^\alpha u_{0h}}\,ds
	+\frac12\int_0^t\iprod{u_h,\fInt^\alpha u_h}\,ds,
\]
so if we define
\[
y(t)=\int_0^t\bigl(\iprod{u_h\,\fInt^\alpha u_h}
	+\|\fInt^\alpha\nabla u_h\|^2\bigr)\,ds,
\]
then
\[
y(t)\le C\int_0^t\iprod{u_{0h},\fInt^\alpha u_{0h}}\,ds
		+C\int_0^t\|\fInt^\alpha u_h\|^2\,ds
			\quad\text{for $0\le t\le T$.}
\]
Noting that $(\fInt^\alpha u_{0h})(t)=u_{0h}\omega_{\alpha+1}(t)$,
and applying Lemma~\ref{lem: y(t)} with~$\phi=u_h$, it
follows that
\begin{equation}\label{eq: y(t) A}
y(t)\le a(t)+b(t)\int_0^t\frac{(t-s)^{\alpha/2-1}}{\Gamma(\alpha/2)}
	\,y(s)\,ds\quad\text{for $0\le t\le T$,}
\end{equation}
where
\[
a(t)=Ct^{\alpha+1}\|u_{0h}\|^2
\quad\text{and}\quad
b(t)=Ct^{\alpha/2}.
\]
Let $E_\beta(z)=\sum_{n=0}^\infty z^n/\Gamma(1+n\beta)$ denote the
Mittag--Leffler function.  A generalised Gronwall inequality of Dixon
and McKee~\cite[Theorem~3.1]{DixonMcKee1986} (also stated in our
earlier paper~\cite[Lemma~2.6]{KimMcLeanMustapha2016}) then yields
\begin{equation}\label{eq: y(t) B}
y(t)\le a(t)E_{\alpha/2}\bigl(b(t)t^{\alpha/2}\bigr)\le Ca(t)
	\quad\text{for $0\le t\le T$.}
\end{equation}
The first estimate of the lemma follows at once, and the second is
then a consequence of~\eqref{eq: fractional D}.\qed
\end{proof}

\begin{lemma}\label{lem: stability B}
For $0\le t\le T$,
\[
\int_0^t\bigl(\iprod{\Mult u_h,\fInt^\alpha\Mult u_h}
	+\|\fInt^\alpha\Mult\nabla u_h\|^2\bigr)\,ds
		\le Ct^{3+\alpha}\|u_{0h}\|^2
\]
and
\[
\int_0^t \|\fInt^\alpha\Mult u_h\|^2\,ds
	\le Ct^{3+2\alpha}\|u_{0h}\|^2.
\]
\end{lemma}
\begin{proof}
We multiply both sides of~\eqref{eq: I FEM} by~$t$, and then use
\eqref{eq: M I}, to obtain
\begin{multline}\label{eq: Muh eqn}
\iprod{\Mult u_h,\chi}
	+\iprod{\fInt^\alpha\Mult\nabla u_h,\nabla\chi}
	+\alpha\iprod{\fInt^{\alpha+1}\nabla u_h,\nabla\chi}\\
	-\iprod{\Mult\vec{B}_1(u_h),\nabla\chi}
	=\iprod{\Mult u_{0h},\chi}.
\end{multline}
By integrating \eqref{eq: I FEM} in time, we find that
\[
\iprod{\fInt^{\alpha+1}\nabla u_h,\nabla\chi}
	=\iprod{\fInt^1(u_{0h}-u_h),\chi}
+\iprod{\fInt^1 \vec{B}_1(u_h), \nabla\chi},
\]
and so, noting that $\fInt^1u_{0h}=\Mult u_{0h}$,
\begin{multline*}
\iprod{\Mult u_h,\chi}
    +\iprod{\fInt^\alpha\Mult\nabla u_h,\nabla\chi}
	=\iprod{\vec{B}_2(u_h),\nabla\chi}
	+\iprod{(1-\alpha)\Mult u_{0h}+\alpha\fInt^1u_h,\chi}\\
	\le\tfrac12\|\vec{B}_2(u_h)\|^2+\tfrac12\|\nabla\chi\|^2
	+\iprod{(1-\alpha)\Mult u_{0h}+\alpha\fInt^1u_h,\chi}.
\end{multline*}
Now choose $\chi=\fInt^\alpha\Mult u_h$, cancel the
term~$\tfrac12\|\nabla\chi\|^2$ and integrate in time to arrive at
the estimate
\begin{multline*}
\int_0^t\bigl(\iprod{\Mult u_h,\fInt^\alpha\Mult u_h}
	+\tfrac12\|\fInt^\alpha\Mult\nabla u_h\|^2\bigr)\,ds\\
	\le\frac12\int_0^t\|\vec{B}_2(u_h)\|^2\,ds
		+\int_0^t\iprod{(1-\alpha)\Mult u_{0h}+\alpha\fInt^1u_h,
			\fInt^\alpha\Mult u_h}\,ds.
\end{multline*}
Using \eqref{eq: fractional A} twice, with~$\epsilon=1/4$, we see that
the second term on the right-hand side is bounded by
\[
\frac12\int_0^t\iprod{\Mult u_h,\fInt^\alpha\Mult u_h}\,ds
	+C\int_0^t\iprod{\Mult u_{0h},\fInt^\alpha\Mult u_{0h}}\,ds
	+C\int_0^t\iprod{\fInt^1u_h,\fInt^\alpha\fInt^1u_h}\,ds
\]
so
\begin{multline*}
\int_0^t\bigl(\iprod{\Mult u_h,\fInt^\alpha\Mult u_h}
	+\|\fInt^\alpha\Mult\nabla u_h\|^2\bigr)\,ds
	\le\int_0^t\|\vec{B}_2(u_h)\|^2\,ds\\
	+C\int_0^t\iprod{\Mult u_{0h},\fInt^\alpha\Mult u_{0h}}\,ds
	+C\int_0^t\iprod{\fInt^1u_h,\fInt^\alpha\fInt^1u_h}\,ds.
\end{multline*}
Since $\fInt^\alpha\Mult u_{0h}
=u_{0h}\fInt^\alpha\omega_2=u_{0h}\omega_{\alpha+2}$, we have
\[
\int_0^t\iprod{\Mult u_{0h},\fInt^\alpha\Mult u_{0h}}\,ds
	=Ct^{3+\alpha}\|u_{0h}\|^2,
\]
and, using \eqref{eq: fractional E} followed
by Lemma~\ref{lem: I nu mu} with $\nu=1$~and $\mu=\alpha$,
\[
\int_0^t\iprod{\fInt^1u_h,\fInt^\alpha\fInt^1 u_h}\,ds
	\le Ct^\alpha\int_0^t\|\fInt^1 u_h\|^2\,ds
	\le Ct^{2-\alpha}\int_0^t\|\fInt^\alpha u_h\|^2\,ds.
\]
Thus, by Lemma~\ref{lem: B integrals},
\begin{multline*}
\int_0^t\bigl(\iprod{\Mult u_h,\fInt^\alpha\Mult u_h}
	+\|\fInt^\alpha\Mult\nabla u_h\|^2\bigr)\,ds
	\le Ct^{3+\alpha}\|u_{0h}\|^2\\
	+C\bigl(t^2+t^{2-\alpha}\bigr)
		\int_0^t\|\fInt^\alpha u_h\|^2\,ds,
\end{multline*}
which, when combined with the second estimate from
Lemma~\ref{lem: stability A}, proves the first claim.
The second follows at once thanks to~\eqref{eq: fractional D}.\qed
\end{proof}

Next, we show that $u_h$ may be replaced with~$(\Mult u_h)'$ in
the first estimate of Lemma~\ref{lem: stability A}.

\begin{lemma}\label{lem: stability C}
For $0\le t\le T$,
\[
\int_0^t\bigl(\iprod{(\Mult u_h)',\fInt^\alpha(\Mult u_h)'}
	+\|\fInt^\alpha(\Mult\nabla u_h)'\|^2\bigr)\,ds
	\le Ct^{1+\alpha}\|u_{0h}\|^2.
\]
\end{lemma}

\begin{proof}
Differentiate \eqref{eq: Muh eqn} to obtain
\[
\iprod{(\Mult u_h)',\chi}
	+\iprod{\partial_t^{1-\alpha}\Mult\nabla u_h,\nabla\chi}
	+\iprod{\alpha\fInt^\alpha\nabla u_h-\vec{B}_3(u_h),\nabla\chi}
	=\iprod{u_{0h},\chi},
\]
and note that
\[
\bigl|
\iprod{\alpha\fInt^\alpha\nabla u_h-\vec{B}_3(u_h),\nabla\chi}
\bigr|\le\tfrac12\|\nabla\chi\|^2
	+\|\vec{B}_3(u_h)\|^2+\alpha^2\|\fInt^\alpha\nabla u_h\|^2.
\]
We choose
$\chi=\partial_t^{1-\alpha}\Mult u_h=(\fInt^\alpha\Mult u_h)'$, and
observe that $(\Mult u_h)(0)=0$ so \eqref{eq: Caputo} implies
$\chi=\fInt^\alpha(\Mult u_h)'$.  Thus,
\begin{multline*}
\iprod{(\Mult u_h)',\fInt^\alpha(\Mult u_h)'}
	+\tfrac12\|\fInt^\alpha(\Mult\nabla u_h)'\|^2\\
	\le\iprod{u_{0h},\fInt^\alpha(\Mult u_h)'}
	+\|\vec{B}_3(u_h)\|^2+\|\fInt^\alpha\nabla u_h\|^2.
\end{multline*}
By~\eqref{eq: fractional A},
\[
\int_0^t \iprod{u_{0h},\fInt^\alpha(\Mult u_h)'}\,ds
	\le\frac12\int_0^t
		\iprod{(\Mult u_h)',\fInt^\alpha(\Mult u_h)'}\,ds
	+C\int_0^t\iprod{u_{0h},\fInt^\alpha u_{0h}}\,ds,
\]
so by Lemma~\ref{lem: B integrals},
\begin{multline*}
y(t):=\int_0^t\bigl(\iprod{(\Mult u_h)',\fInt^\alpha(\Mult u_h)'}
	+\|\fInt^\alpha(\Mult\nabla u_h)'\|^2\bigr)\,ds
	\le C\int_0^t\iprod{u_{0h},\fInt^\alpha u_{0h}}\,ds\\
	 +C\int_0^t\bigl(\|\fInt^\alpha\nabla u_h\|^2
		+\|\fInt^\alpha\Mult u_h\|^2
		+\|\fInt^\alpha u_h\|^2\bigr)\,ds
		+C\int_0^t\|\fInt^\alpha(\Mult u_h)'\|^2\,ds.
\end{multline*}
The first integral on the right-hand side is bounded
by~$Ct^{1+\alpha}\|u_{0h}\|^2$, and so is the second via
Lemmas \ref{lem: stability A}~and \ref{lem: stability B}.  It follows
using Lemma~\ref{lem: y(t)} that $y(t)$ satisfies an inequality of the
form~\eqref{eq: y(t) A} with $a(t)=Ct^{1+\alpha}\|u_{0h}\|^2$~and
$b(t)=Ct^{\alpha/2}$, so \eqref{eq: y(t) B} holds, proving the result.
\qed
\end{proof}

The stability of~$u_h(t)$ in~$L_2(\Omega)$ now follows.

\begin{theorem}\label{thm: stability}
There is a constant~$C$, depending on $\alpha$, $T$~and $\vec{F}$,
such that
\[
\|u_h(t)\|\le C\|u_{0h}\|\quad\text{for $0\le t\le T$.}
\]
\end{theorem}
\begin{proof}
Using Lemma~\ref{lem: phi(t)-phi(0)} with~$\phi=\Mult u_h$, followed
by Lemma~\ref{lem: stability C}, we obtain
\[
t^2\|u_h(t)\|^2=
\|(\Mult u_h)(t)\|^2\le Ct^{1-\alpha}
	\int_0^t\bigiprod{(\Mult u_h)',\fInt^\alpha(\Mult u_h)'}\,ds
	\le Ct^2\|u_{0h}\|^2.
\]
\qed
\end{proof}

{ Because some of the estimates of Section~\ref{sec: fractional}
break down as~$\alpha\to1$, the same is true of the stability result
above.  That is, the proof of Theorem~\ref{thm: stability} yields a
constant~$C$ that tends to infinity as~$\alpha\to1$.  However, we
can easily prove stability in the limiting case when~$\alpha=1$, that
is, when \eqref{eq: FFP} reduces to the classical Fokker--Planck
equation,
\[
\partial_tu+\nabla\cdot(\nabla u-\vec{F} u)=0,
\]
and the finite element equation~\eqref{eq: FEM} to
\[
\iprod{\partial_tu_h,\chi}+\iprod{\nabla u_h,\nabla\chi}
	-\iprod{\vec{F} u_h,\nabla\chi}=0.
\]}
%Choosing $\chi=u_h(t)$ gives
%\[
%\partial_t\bigl(\tfrac12\|u_h(t)\|^2\bigr)+\|\nabla u_h\|^2
%	=\iprod{\vec{F} u_h,\nabla u_h}
%	\le\tfrac12\|\vec{F} u_h\|^2+\tfrac12\|\nabla u_h\|^2,
%\]
%so
%\[
%\|u_h(t)\|^2-\|u_{0h}\|^2\le\int_0^t\|\vec{F} u_h\|^2\,ds
%	\le C\int_0^t\|u_h\|^2\,ds,
%\]
%and the classical Gronwall inequality implies that
%$\|u_h(t)\|\le C\|u_{0h}\|$ for~$0\le t\le T$, where $C$ depends on
%$T$~and $\vec{F}$.
%%%%%%%%%%%%%%%%%%%%%%%%%%%%%%%%%%%%%%%%%%%%%%%%%%%%%%%%%%%%%%%%%%%%%%
\section{Error estimate}\label{sec: error}
We now seek to estimate the accuracy of the semidiscrete finite
element solution~$u_h$. Recall that the Ritz projection~$R_hv\in\Sh$
of a function~$v\in H^1(\Omega)$ is defined by
\[
\iprod{\nabla R_hv,\nabla\chi}+\iprod{R_hv,\chi}
	=\iprod{\nabla v,\nabla\chi}+\iprod{v,\chi}
	\quad\text{for all $\chi\in\Sh$;}
\]
here, the lower-order terms are included to allow for a
zero-flux boundary condition~\eqref{eq: Neuman bc}, in which case
the functions in~$\Sh$ do not have to vanish on~$\partial\Omega$ and
so the Poincar\'e inequality is not applicable.  Since the Galerkin
finite element method is quasi-optimal in~$H^1(\Omega)$, we know that
$\|v-R_hv\|_1\le Ch\|v\|_2$ for $v\in H^2(\Omega)$. Assuming that
$\Omega$ is convex, so that the Poisson problem is $H^2$-regular, the
usual duality argument implies that
\begin{equation}\label{eq: R_h error}
\|v-R_hv\|\le Ch^2\|v\|_2\quad\text{for $v\in H^2(\Omega)$.}
\end{equation}

We now decompose the error into
\begin{equation}\label{eq: eh}
e_h=u_h-u=\theta_h-\rho_h\quad\text{where}\quad
\theta_h=u_h-R_hu\quad\text{and}\quad\rho_h=u-R_hu,
\end{equation}
and deduce from \eqref{eq: FFP weak}~and \eqref{eq: FEM} that
\begin{equation}\label{eq: theta_h}
\iprod{\theta_h',\chi}
	+\iprod{\partial_t^{1-\alpha}\nabla\theta_h,\nabla\chi}
	-\iprod{\vec{F}\partial_t^{1-\alpha}\theta_h,\nabla\chi}
	=\iprod{\rho_h'-\partial_t^{1-\alpha}\rho_h,\chi}
	-\iprod{\vec{F}\partial_t^{1-\alpha}\rho_h,\nabla\chi}.
\end{equation}
With this equation, we can use the techniques of
Section~\ref{sec: stability} to estimate $\theta_h$ in terms
of~$\rho_h$. The next lemma provides our basic estimate for the
latter.

\begin{lemma}\label{lem: I beta rho}
Let $\beta\ge0$~and $0\le r\le2$. If $u$ has the regularity
property~\eqref{eq: u reg}, then
\[
\|\fInt^\beta\rho_h\|+\|\fInt^\beta(\Mult\rho_h')\|
	\le Ct^{\beta+\alpha(r-2)/2}h^2K_r\quad\text{for $0<t\le T$.}
\]
\end{lemma}
\begin{proof}
For the case~$\beta=0$, we see from~\eqref{eq: R_h error} that
\[
\|\rho_h(t)\|+\|\Mult\rho_h'(t)\|\le
	Ch^2\bigl(\|u(t)\|_2+t\|u'(t)\|_2\bigr)
	\le Ct^{\alpha(r-2)/2}h^2 K_r,
\]
whereas for~$\beta>0$,
\begin{align*}
\|\fInt^\beta\rho_h(t)\|+\|\fInt^\beta(\Mult\rho'_h)\|
	&\le\int_0^t\omega_\beta(t-s)
		\bigl(\|\rho_h(s)\|+s\|\rho_h'(s)\|\bigr)\,ds\\
	&\le C\int_0^t(t-s)^{\beta-1}\,
		s^{\alpha(r-2)/2}h^2K_r\,ds,
\end{align*}
and the result follows after making the substitution~$s=ty$
for~$0\le y\le 1$. \qed
\end{proof}

The proofs of Lemmas \ref{lem: error A}~and \ref{lem: error B} below
parallel those of Lemmas \ref{lem: stability A}~and
\ref{lem: stability B} from Section~\ref{sec: stability}.  We let
$P_h$ denote $L_2$-projector onto the finite element subspace~$\Sh$,
that is, for any~$v\in L_2(\Omega)$ we define $P_hv\in\Sh$
by~$\iprod{P_hv,\chi}=\iprod{v,\chi}$ for all~$\chi\in\Sh$.

\begin{lemma}\label{lem: error A}
If $u_{0h}=P_hu_0$ then, for $0\le t\le T$~and $0\le r\le2$,
\[
\int_0^t\bigl(\iprod{\theta_h,\fInt^\alpha\theta_h}
	+\|\fInt^\alpha\nabla\theta_h\|^2\bigr)\,ds\le Ct^{1+\alpha(r-1)}
		h^4K_r^2
\]
and
\[
\int_0^t\|\fInt^\alpha\theta_h\|^2\,ds
	\le Ct^{1+\alpha r}h^4K_r^2.
\]
\end{lemma}
\begin{proof}
We integrate \eqref{eq: theta_h} in time to obtain
\begin{equation}\label{eq: I theta_h}
\iprod{\theta_h,\chi}
	+\iprod{\fInt^\alpha\nabla\theta_h,\nabla\chi}
	-\iprod{\vec{B}_1(\theta_h),\nabla\chi}
	=\iprod{e_h(0),\chi}
	+\iprod{\tilde\rho_h,\chi}
	-\iprod{\vec{B}_1(\rho_h),\nabla\chi},
\end{equation}
where $\tilde\rho_h=\rho_h-\fInt^\alpha\rho_h$.
Our choice of~$u_{0h}$ means that $\iprod{e_h(0),\chi}=0$, so by
letting $\chi=\fInt^\alpha\theta_h$ and recalling the
definitions~\eqref{eq: M B}, we see that
\[
\iprod{\theta_h,\fInt^\alpha\theta_h}
	+\|\fInt^\alpha\nabla\theta_h\|^2
\le\|\vec{B}_1(\theta_h)\|^2+\|\vec{B}_1(\rho_h)\|^2
	+\tfrac12\|\fInt^\alpha\nabla\theta_h\|^2
	+\iprod{\tilde\rho_h,\fInt^\alpha\theta_h}.
\]
Thus, by Lemma~\ref{lem: B integrals},
\begin{multline*}
\int_0^t\bigl(\iprod{\theta_h,\fInt^\alpha\theta_h}
	+\tfrac12\|\fInt^\alpha\nabla\theta_h\|^2\bigr)\,ds
	\le C\int_0^t\|\fInt^\alpha\theta_h\|^2\,ds\\
	+C\int_0^t\|\fInt^\alpha\rho_h\|^2\,ds
+\int_0^t\iprod{\tilde\rho_h,\fInt^\alpha\theta_h}\,ds.
\end{multline*}
After applying \eqref{eq: fractional A} with $\phi=\tilde\rho_h$~and
$\psi=\theta_h$, followed by Lemma~\ref{lem: y(t)}
with~$\phi=\theta_h$, we see that the function
\[
y(t)=\int_0^t\bigl(\iprod{\theta_h,\fInt^\alpha\theta_h}
	+\|\fInt^\alpha\nabla\theta_h\|^2\bigr)\,ds
\]
satisfies an inequality of the form~\eqref{eq: y(t) A} with
\[
a(t)=C\int_0^t\iprod{\tilde\rho_h,\fInt^\alpha\tilde\rho_h}\,ds
	+C\int_0^t\|\fInt^\alpha\rho_h\|^2\,ds
\quad\text{and}\quad
b(t)=Ct^{\alpha/2}.
\]
For brevity, put $\eta=h^2K_r$.  By Lemma~\ref{lem: I beta rho},
\[
\bigl|\iprod{\tilde\rho_h,\fInt^\alpha\tilde\rho_h}\bigr|
	\le C\eta^2(1+t^\alpha)t^{\alpha(r-2)/2}
		(1+t^\alpha)t^{\alpha+\alpha(r-2)/2}
	\le C\eta^2t^{\alpha(r-1)}
\]
and $\|\fInt^\alpha\rho_h\|^2\le
C\bigl(\eta t^{\alpha+\alpha(r-2)/2}\bigr)^2=C\eta^2t^{\alpha r}$, so
$a(t)\le C\eta^2t^{\alpha(r-1)+1}$.  Thus, the two estimates follow
from~\eqref{eq: y(t) B} followed by~\eqref{eq: fractional D}.\qed
\end{proof}

\begin{lemma}\label{lem: error B}
If $u_{0h}=P_hu_0$ then, for $0\le t\le T$ and $0\le r\le2$,
\[
\int_0^t\bigl(\iprod{\Mult\theta_h,\fInt^\alpha\Mult\theta_h}
	+\|\fInt^\alpha\Mult\nabla\theta_h\|^2\bigr)\,ds
	\le Ct^{3+\alpha(r-1)}h^4K_r^2
\]
and
\[
\int_0^t\|\fInt^\alpha\Mult\theta_h\|^2\,ds
	\le Ct^{3+\alpha r}h^4K_r^2.
\]
\end{lemma}
\begin{proof}
We multiply both sides of~\eqref{eq: I theta_h} by~$t$, remembering
that $\iprod{e_h(0),\chi}=0$, and then use \eqref{eq: M I} to obtain
\begin{multline}\label{eq: M theta_h eqn}
\iprod{\Mult\theta_h,\chi}
	+\iprod{\fInt^\alpha\Mult\nabla\theta_h,\nabla\chi}
	+\alpha\iprod{\fInt^{\alpha+1}\nabla\theta_h,\nabla\chi}
	-\iprod{\Mult\vec{B}_1(\theta_h),
		\nabla\chi}\\
	=\iprod{\Mult\tilde\rho_h,\chi}
-\iprod{\Mult \vec{B}_1(\rho_h),\nabla\chi}.
\end{multline}
By integrating~\eqref{eq: I theta_h}, we find that
\[
\iprod{\fInt^{\alpha+1}\nabla\theta_h,\nabla\chi}
	=\iprod{\fInt^1\tilde\rho_h-\fInt^1\theta_h,\chi}
	+\iprod{\fInt^1\vec{B}_1(\theta_h)
	-\fInt^1\vec{B}_1(\rho_h),\nabla\chi},
\]
and hence, with~$\vec{B}_2(\phi)$ defined as before
in~\eqref{eq: M B},
\begin{multline*}
\iprod{\Mult\theta_h,\chi}
	+\iprod{\fInt^\alpha\Mult\nabla\theta_h,\nabla\chi}
	=\iprod{\vec{B}_2(\theta_h)-\vec{B}_2(\rho_h),\nabla\chi}\\
	+\iprod{(\Mult-\alpha\fInt^1)\tilde\rho_h
		+\alpha\fInt^1\theta_h,\chi}.
\end{multline*}
Now choose $\chi=\fInt^\alpha\Mult\theta_h$ so that, after cancelling
a term~$\tfrac12\|\nabla\chi\|^2$ and integrating,
\begin{multline*}
\int_0^t\bigl(\iprod{\Mult\theta_h,\fInt^\alpha\Mult\theta_h}
	+\tfrac12\|\fInt^\alpha\Mult\nabla\theta_h\|^2\bigr)\,ds
	\le\frac12\int_0^t\|\vec{B}_2(\theta_h)-\vec{B}_2(\rho_h)\|^2\,ds\\
	+\int_0^t\iprod{(\Mult-\alpha\fInt^1)\tilde\rho_h
		+\alpha\fInt^1\theta_h,\fInt^\alpha\Mult\theta_h}\,ds.
\end{multline*}
Using \eqref{eq: fractional A} with~$\epsilon=1/4$,
$\phi=(\Mult-\alpha\fInt^1)\tilde\rho_h$ and $\psi=\Mult\theta_h$,
and a second time with~$\phi=\alpha\fInt^1\theta_h$, we see that
\begin{multline*}
\int_0^t\bigl(\iprod{\Mult\theta_h,\fInt^\alpha\Mult\theta_h}
	+\|\fInt^\alpha\Mult\nabla\theta_h\|^2\bigr)\,ds
	\le\int_0^t\|\vec{B}_2(\theta_h)-\vec{B}_2(\rho_h)\|^2\,ds\\
	+C\int_0^t\iprod{(\Mult-\alpha\fInt^1)\tilde\rho_h,\fInt^\alpha
	(\Mult-\alpha\fInt^1)\tilde\rho_h}\,ds
	+C\int_0^t\iprod{\fInt^1\theta_h,\fInt^\alpha\fInt^1\theta_h}\,ds.
\end{multline*}
Lemma~\ref{lem: B integrals} implies that
\[
\int_0^t\|\vec{B}_2(\theta_h)-\vec{B}_2(\rho_h)\|^2\,ds
	\le Ct^2\int_0^t\bigl(\|\fInt^\alpha\theta_h\|^2
		+\|\fInt^\alpha\rho_h\|^2\bigr)\,ds
\]
and, putting $\eta=h^2K_r$ as before, we find with the help
of Lemma~\ref{lem: I beta rho} that
\[
\int_0^t
\bigl|\iprod{(\Mult-\alpha\fInt^1)\tilde\rho_h,\fInt^\alpha
	(\Mult-\alpha\fInt^1)\tilde\rho_h}\bigr|
	\le C\eta^2 t^{3+\alpha(r-1)}.
\]
Using \eqref{eq: fractional E}, followed by
Lemma~\ref{lem: I nu mu} with $\nu=1$~and $\mu=\alpha$,
\[
\int_0^t\iprod{\fInt^1\theta_h,\fInt^\alpha\fInt^1\theta_h}\,ds
	\le Ct^\alpha\int_0^t\|\fInt^1\theta_h\|^2\,ds
	\le Ct^{2-\alpha}\int_0^t\|\fInt^\alpha\theta_h\|^2\,ds,
\]
so, recalling that
$\|\fInt^\alpha\rho_h\|^2\le C\eta^2t^{\alpha r}$, the first
estimate follows by Lemma~\ref{lem: error A}.  The second is then an
immediate consequence of~\eqref{eq: fractional D}.
\end{proof}

Techniques like those of Lemma~\ref{lem: stability C}~and
Theorem~\ref{thm: stability} now yield our error bound.

\begin{theorem}\label{thm: error}
If $\Omega$ is convex and the solution of the fractional
Fokker--Planck equation~\eqref{eq: FFP} has the regularity
property~\eqref{eq: u reg}, then the finite element solution, given
by~\eqref{eq: FFP weak}, satisfies
\[
\|u_h(t)-u(t)\|\le C\|u_{0h}-P_hu_0\|+Ct^{\alpha(r-2)/2}h^2K_r
\]
for $0<t\le T$ and $0\le r\le2$. The constant~$C$ may depend on
$\alpha$, $T$~and $\vec{F}$.
\end{theorem}
\begin{proof}
Suppose in the first instance that $u_{0h}=P_hu_0$, as required for
Lemmas \ref{lem: error A}~and \ref{lem: error B}.
Differentiate~\eqref{eq: M theta_h eqn} to obtain
\begin{multline*}
\iprod{(\Mult\theta_h)',\chi}
	+\iprod{\partial_t^{1-\alpha}\Mult\nabla\theta_h,\nabla\chi}
	+\alpha\iprod{\fInt^\alpha\nabla\theta_h,\nabla\chi}\\
	=\iprod{(\Mult\tilde\rho_h)',\chi}
	+\iprod{\vec{B}_3(\theta_h)-\vec{B}_3(\rho_h), \nabla\chi},
\end{multline*}
where $\vec{B}_3(\phi)$ is again defined as in~\eqref{eq: M B}.
Noting that
\begin{align*}
\bigl|\iprod{\vec{B}_3(\theta_h)-\vec{B}_3(\rho_h)
	-\alpha\fInt^\alpha\theta_h, \nabla\chi}\bigr|
&\le\|\nabla\chi\|^2
	+\tfrac12\bigl(\|\vec{B}_3(\theta_h)-\vec{B}_3(\rho_h)\|^2\\
	&\qquad{}+\tfrac12\alpha^2\|\fInt^\alpha\nabla\theta_h\|^2\bigr),
\end{align*}
we choose $\chi=\partial_t^{1-\alpha}\Mult\theta_h
=(\fInt^\alpha\Mult\theta_h)'$, and observe that
$(\Mult\theta_h)(0)=0$ so \eqref{eq: Caputo} implies
$\chi=\fInt^\alpha(\Mult\theta_h)'$.  Thus, after cancelling
$\|\nabla\chi\|^2$,
\begin{align*}
\iprod{(\Mult\theta_h)',\fInt^\alpha(\Mult\theta_h)'}&\le
	\iprod{(\Mult\tilde\rho_h)',\fInt^\alpha(\Mult\theta_h)'}\\
	&\qquad{}+\tfrac12\|\vec{B}_3(\theta_h)-\vec{B}_3(\rho_h)\|^2
		+\tfrac12\alpha^2\|\fInt^\alpha\nabla\theta_h\|^2.
\end{align*}
Integrating in time, and then applying~\eqref{eq: fractional A}
to the first term on the right hand side, with $\epsilon=1/2$,
$\phi=(\Mult\tilde\rho_h)'$ and $\psi=(\Mult\theta_h)'$, it follows
that
\begin{multline*}
\int_0^t\bigiprod{(\Mult\theta_h)',\fInt^\alpha(\Mult\theta_h)'}\,ds
	\le C\int_0^t\bigiprod{(\Mult\tilde\rho_h)',
		\fInt^\alpha(\Mult\tilde\rho_h)'}\,ds\\
		+\int_0^t\bigl(\|\vec{B}_3(\theta_h)-\vec{B}_3(\rho_h)\|^2
			+\|\fInt^\alpha\nabla\theta_h\|^2\bigr)\,ds.
\end{multline*}
Since, using \eqref{eq: M I},
\begin{align*}
(\Mult\tilde\rho_h)'&=\bigl[\Mult(\rho_h-\fInt^\alpha\rho_h)\bigr]'
	=\rho_h+\Mult\rho'_h-\bigl[\fInt^\alpha\Mult\rho_h
		+\alpha\fInt^{\alpha+1}\rho_h\bigr]'\\
	&=\rho_h+\Mult\rho'_h-\fInt^\alpha(\Mult\rho_h)'
		-\alpha\fInt^\alpha\rho_h\\
	&=\rho_h+\Mult\rho'_h-\fInt^\alpha\Mult\rho_h'
		-(1+\alpha)\fInt^\alpha\rho_h
\end{align*}
we see from \eqref{eq: R_h error}, \eqref{eq: u reg}~and
Lemma~\ref{lem: I beta rho} that
$\|(\Mult\tilde\rho_h)'\|\le C\eta t^{\alpha(r-2)/2}(1+t^\alpha)$
where, as before, $\eta=h^2K_r$.  Consequently,
\[
\int_0^t\bigiprod{(\Mult\tilde\rho_h)',
		\fInt^\alpha(\Mult\tilde\rho_h)'}\,ds
	\le C\eta^2 t^{1+\alpha(r-1)},
\]
and by Lemma~\ref{lem: B integrals},
\begin{align*}
\int_0^t\|\vec{B}_3(\rho_h)\|^2\,ds
	&\le C\int_0^t\bigl(\|\fInt^\alpha(\Mult\rho_h)'\|^2
	+\|\fInt^\alpha(\Mult\rho_h)\|^2+\|\fInt^\alpha\rho_h\|^2
		\bigr)\,ds\\
	&\le C\eta^2\int_0^t\bigl(t^{\alpha r}
		+t^{2+\alpha r}+t^{\alpha r} \bigr)\,ds
	\le C\eta^2t^{1+\alpha r},
\end{align*}
showing that
\begin{multline*}
\int_0^t\bigiprod{(\Mult\theta_h)',\fInt^\alpha(\Mult\theta_h)'}\,ds
	\le C\eta^2t^{1+\alpha(r-1)}\\
	+C\int_0^t\bigl(\|\fInt^\alpha\nabla\theta_h\|^2
		+\|\fInt^\alpha\Mult\theta_h\|^2+\|\fInt^\alpha\theta_h\|^2
		\bigr)\,ds+C\int_0^t\|\fInt^\alpha(\Mult\theta_h)'\|^2\,ds.
\end{multline*}
Using Lemmas \ref{lem: error A}~and \ref{lem: error B}, we find that
the second term on the right is bounded by
$Ct^{1+\alpha(r-1)}\eta^2$.  It follows using
Lemma~\ref{lem: y(t)} that the function
\[
y(t)=\int_0^t
	\bigiprod{(\Mult\theta_h)',\fInt^\alpha(\Mult\theta_h)'}\,ds
\]
satisfies an inequality of the form~\eqref{eq: y(t) A} with
$a(t)=Ct^{1+\alpha(r-1)}\eta^2$~and $b(t)=Ct^{\alpha/2}$.
Therefore, using Lemma~\ref{lem: phi(t)-phi(0)}
with~$\phi=\Mult\theta_h$, followed by~\eqref{eq: y(t) B}, we have
\[
\|\Mult\theta_h\|^2\le Ct^{1-\alpha}y(t)\le Ct^{1-\alpha}a(t)
	\le Ct^{2+\alpha(r-2)}\eta^2,
\]
which is equivalent to the estimate
$\|\theta_h\|\le Ct^{\alpha(r-2)/2}h^2K_r$.  Recalling
\eqref{eq: eh}, the desired error bound in the
case~$u_{0h}=P_hu_0$ follows by the triangle inequality and the
case~$\beta=0$ of Lemma~\ref{lem: I beta rho}.

The error bound for general~$u_{0h}$ now follows from the stability
result of Theorem~\ref{thm: stability}.  In fact, if $u_h^*$~and
$u_h$ denote the finite element solutions satisfying
$u_h^*(0)=P_hu_0$~and $u_h(0)=u_{0h}$, then the difference
$u_h-u^*_h$ is the finite element solution with initial
value~$u_{0h}-P_hu_0$ so
\[
\|u_h(t)-u^*_h(t)\|\le C\|u_{0h}-P_hu_0\|
	\quad\text{for $0\le t\le T$.}
\]
We obtain the desired estimate for~$\|u_h(t)-u(t)\|$ after applying
the triangle inequality, noting that
$\|u_h^*(t)-u(t)\|\le Ct^{\alpha(r-2)/2}h^2K_r$.\qed
\end{proof}

If $r<2$, then the error estimate in the theorem becomes unbounded
as~$t\to0$, but the stability result of Theorem~\ref{thm: stability}
shows that the error must in fact remain bounded.
%%%%%%%%%%%%%%%%%%%%%%%%%%%%%%%%%%%%%%%%%%%%%%%%%%%%%%%%%%%%%%%%%%%%%%
\section{Numerical examples}\label{sec: numerical}
We discuss experiments with two problems, using a fully-discrete
scheme of implicit Euler type. For time levels
$0=t_0<t_1<t_2<\cdots<t_N=T$, we denote the
$n$th step size by~$k_n=t_n-t_{n-1}$ and the associated subinterval
by~$I_n=(t_{n-1},t_n)$, for $1\le n\le N$.  The maximum step
size~$k=\max_{1\le n\le N}k_n$ is sometimes used to label quantities
that depend on the mesh. With any sequence of values $V^1$, $V^2$,
\dots, $V^N$ we associate the piecewise-constant function $\check V$
defined by
\[
\check V(t)=V^n\quad\text{for $t_{n-1}<t<t_n$ and $n\ge1$.}
\]
Integrating the finite element equation~\eqref{eq: FEM}
over the $n$th time interval~$I_n$ gives
\[
\bigiprod{u_h(t_n)-u_h(t_{n-1}),\chi}
	+\int_{I_n}\bigiprod{\partial_t^{1-\alpha}\nabla u_h,
		\nabla\chi}\,dt
	-\int_{I_n}\bigiprod{\vec{F}\partial_t^{1-\alpha} u_h,
		\nabla\chi}\,dt=0,
\]
for all~$\chi\in\Sh$, and we approximate~$u_h(t_n)$ by~$U^n_h\in\Sh$
satisfying
\begin{equation}\label{eq: Unh}
\bigiprod{U^n_h-U^{n-1}_h,\chi}
	+\int_{I_n}\bigiprod{\partial_t^{1-\alpha}\nabla\check U_{h},
		\nabla\chi}\,dt
	-\int_{I_n}\bigiprod{\check{\vec{F}}\partial_t^{1-\alpha}\check U_h,
		\vec{\nabla}\chi}\,dt=0,
\end{equation}
for all $\chi\in\Sh$~and for $1\le n\le N$, with~$U^0_h=u_{0h}$.
For $1\le p\le Q_h:=\dim\Sh$, let $\vec{x}_p$ denote the $p$th free
node of the spatial mesh, and let $\phi_p\in\Sh$ denote the $p$th
nodal basis function, so that $\phi_p(\vec{x}_q)=\delta_{pq}$~and
\[
U^n_h(\vec{x})=\sum_{p=1}^{Q_h}U^n_p\phi_p(\vec{x})\quad
	\text{where}\quad
U^n_p=U^n_h(\vec{x}_p)\approx u_h(\vec{x}_p,t_n)\approx u(\vec{x}_p,t_n).
\]
We define $Q_h\times Q_h$ matrices $\vec{M}$~and $\vec{G}^n$ with
entries
\[
M_{pq}=\iprod{\phi_q,\phi_p}\quad\text{and}\quad
G^n_{pq}=\iprod{\nabla
\phi_{q},\nabla\phi_p}-\iprod{\vec{F}^n\phi_q,\nabla\phi_p},
\]
where $\vec{F}^n(\vec{x})=\vec{F}(\vec{x},t_n)$, and the
$Q_h$-dimensional column vector $\vec{U}^n$ with components~$U^n_p$.
It follows
from~\eqref{eq: Unh} that
\[
\vec{M}\vec{U}^n-\vec{M}\vec{U}^{n-1}
	+\sum_{j=1}^n\omega_{nj}\vec{G}^n\vec{U}^j
	-\sum_{j=1}^{n-1}\omega_{n-1,j}\vec{G}^n\vec{U}^j=0
	\quad\text{for~$1\le n\le N$,}
\]
with weights~$\omega_{nj}=\int_{I_j}\omega_\alpha(t_n-s)\,ds$
for~$1\le j\le n\le N$.
Thus, at the $n$th time step we must solve the linear system
\[
\bigl(\vec{M}+\omega_{nn}\vec{G}^n\bigr)\vec{U}^n
	=\vec{M}\vec{U}^{n-1}
	-\sum_{j=1}^{n-1}\bigl(\omega_{nj}-\omega_{n-1,j}\bigr)\vec{G}^n
	\vec{U}^j.
\]
Although this fully-discrete scheme lacks a theoretical error
analysis, we observed numerically that first-order accuracy in time is
achieved, for~$t$ bounded away from zero, if we use a graded mesh of
the form
\begin{equation}\label{eq: graded}
t_n = (n/N)^\gamma T
	\quad\text{for $0\le n\le N$, with $\gamma=1/\alpha$.}
\end{equation}
Our earlier paper~\cite[Table~5.3]{KimMcLeanMustapha2016}
includes computations with \emph{smooth} initial data, in which we
observed that the $L_2$ error is $O(h^2)$ uniformly for~$0\le t\le T$,
consistent with Theorem~\ref{thm: error} when~$r=2$.  Here, we instead
focus on the case of \emph{non}-smooth initial data.

\begin{figure}
\begin{center}
\includegraphics[scale=0.5]{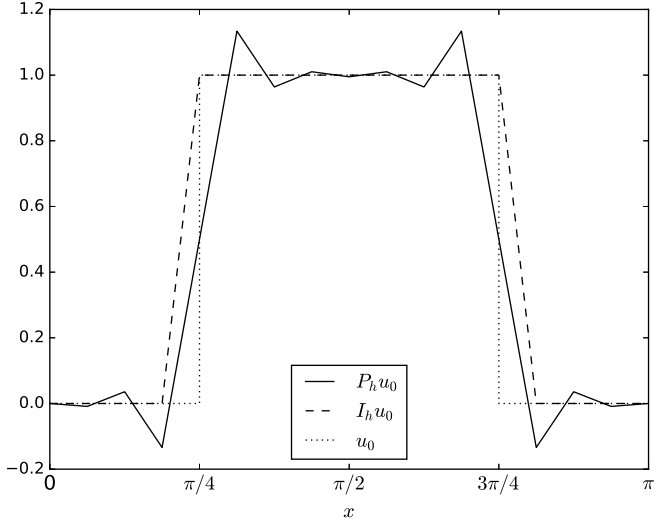}
\end{center}
\caption{The $L_2$-projection~$P_hu_0$ and the nodal
interpolant~$I_h u_{0h}$ of the discontinuous
initial data~\eqref{eq: discontinuous} when~$Q_h=15$.}
\label{fig: u0}
\end{figure}

\subsection{Dirichlet boundary condition}
In our first example, $F(x,t) = -x + \sin t$,
$T=1$~and $\Omega = (0,\pi)$,
with homogeneous Dirichlet boundary conditions
$u(0,t)=0=u(\pi,t)$ and discontinuous initial data given by
\begin{equation}\label{eq: discontinuous}
u_0(x) =
\begin{cases}
 1,\quad x\in[\pi/4,3\pi/4]\\
 0,\quad x\in[0,\pi/4)\cup(3\pi/4,1];
\end{cases}
\end{equation}
Figure~\ref{fig: u0} shows $u_0$~and its $L_2$-projection~$P_hu_0$, as
well as the nodal interpolant~$I_h u_0\in\Sh$ defined by
\begin{equation}\label{eq: nodal}
I_h u_0(x_p) =
\begin{cases}
 1,\quad x_p\in[\pi/4,3\pi/4]\\
 0,\quad x_p\in[0,\pi/4)\cup(3\pi/4,1].
\end{cases}
\end{equation}
The Dirichlet eigenvalues and orthonormal eigenfunctions of
$-\nabla^2=-\partial_x^2$ are
\[
\lambda_m=m^2\quad\text{and}\quad
\varphi_m(x)=\biggl(\frac{2}{\pi}\biggr)^{1/2}\,\sin m x
\quad\text{for $m\in\{1,2,3,\ldots\}$,}
\]
so for $0\le r<1/2$ we have
\[
\|u_0\|_r^2=\sum_{m=1}^\infty m^{2r}\iprod{u_0,\varphi_m}^2
	=\frac{4}{\pi}\sum_{j=1}^\infty(2j-1)^{2(r-1)}
	\le\frac{C}{1-2r}.
\]
If our conjecture that $K_r=C\|u_0\|_r$ in~\eqref{eq: u reg} is
valid, then applying Theorem~\ref{thm: error} with
$r=\tfrac12-\epsilon$~and $\epsilon^{-1}=\log(e^2+t^{-1})$, so that
$t^{-\epsilon}\le e$~and $0<\epsilon<1/2$, gives
\begin{equation}\label{eq: error ex1}
\|u_h(t)-u(t)\|\le C\|u_{0h}-P_hu_0\|+
	Ct^{-3\alpha/4}h^2\sqrt{\log(e^2+t^{-1})}
	\quad\text{for $0<t\le1$.}
\end{equation}

In our computations, we employed nonuniform time levels given
by~\eqref{eq: graded}, but a uniform spatial mesh with~$h=1/(Q_h+1)$.
In all cases, $Q_h+1$ was divisible by~$4$ so that the
points~$\pi/4$~and $3\pi/4$ (where $u_0$ is discontinuous)
coincided with two of the nodes. We first computed a reference
solution~$\Uref^n=U^n_h$ using a fine mesh with $N=10,000$~and
$Q_h=511$. We then computed $U^n_h$ for~$Q_h\in\{7,15,31,63\}$, again
with~$N=10,000$.  The initial data was chosen as~$u_{0h}=P_hu_0$ in
each case.   With such a small~$k$, the error,
\[
E_{h,k}^n=\|U^n_h-\Uref^n\|
	\quad\text{for $1\le n\le N$,}
\]
was dominated by the influence of the spatial discretisation, and we
sought to estimate the convergence rates $\sigma_{h,k}$ such that
\begin{equation}\label{eq: E*}
E^*_{h,k}=\max_{0\leq n\leq N}
	\frac{t_n^{3\alpha/4}E_{h,k}^n}{\sqrt{\log(e^2+t_n^{-1})}}
	\approx Ch^{\sigma_{h,k}},
\end{equation}
from the relation
\begin{equation}\label{eq: rate}
\sigma_{h,k}=\log_2(E^*_{2h,k}/E^*_{h,k}).
\end{equation}
Table~\ref{table: errors} shows the values of $E^*_{h,k}$~and
$\sigma_{h,k}$ for three different values of~$\alpha$. The computed
values of~$\sigma_{h,k}$ are close to~$2$, as expected from
Theorem~\ref{thm: error}.  Figure~\ref{fig: error} shows how the
$L_2$-error $E_{h,k}^n$ varies with~$t_n$ for different~$h$
when~$\alpha=0.75$, again keeping $N=10,000$.  Due to
the log-log scale, the graph of a function proportional
to~$t^{-3\alpha/4}$ appears as a straight line with
gradient~$-3\alpha/4$, indicated by the small triangle, and we observe
exactly this behaviour of the error for~$t$ close---but not \emph{too}
close---to zero.

\begin{table}
\caption{Weighted errors~\eqref{eq: E*}~and
convergence rates~\eqref{eq: rate} for different $\alpha$, when
$u_{0h}=P_hu_0$.}
\label{table: errors}
\begin{center}
\renewcommand{\arraystretch}{1.2}
\begin{tabular}{c|cc|cc|cc}
$Q_h$&
\multicolumn{2}{c|}{$\alpha=0.25$}&
\multicolumn{2}{c|}{$\alpha=0.50$}&
\multicolumn{2}{c}{$\alpha=0.75$}\\
\hline
   7& 7.98e-03&       & 7.77e-03&       & 7.84e-03&       \\
  15& 1.96e-03&  2.024& 1.91e-03&  2.024& 1.94e-03&  2.017\\
  31& 4.88e-04&  2.008& 4.75e-04&  2.008& 4.82e-04&  2.007\\
  63& 1.21e-04&  2.014& 1.18e-04&  2.014& 1.19e-04&  2.015
\end{tabular}
\end{center}
\end{table}

\begin{figure}
\begin{center}
\includegraphics[scale=0.5]{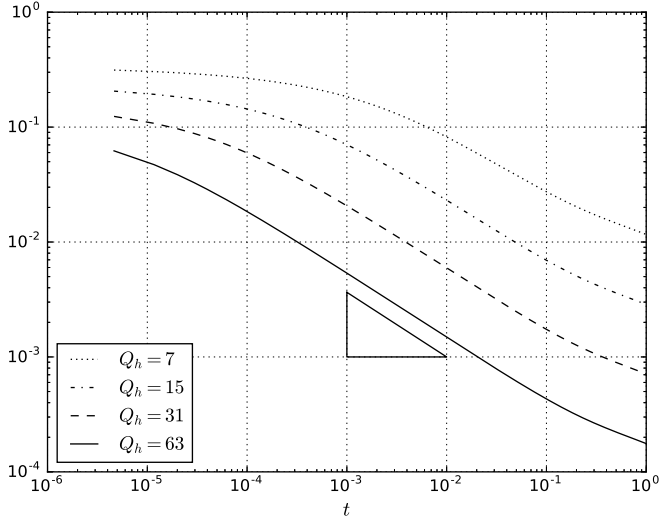}
\caption{Plots of the error $E^n_{h,k}$ as a function of~$t_n$,
for~$\alpha=0.75$ and different choices of~$Q_h$.
The triangle indicates the gradient~$-3\alpha/4$ for a function
proportional to~$t^{-3\alpha/4}$; cf.~\eqref{eq: error ex1}. Note
the logarithmic scales.} \label{fig: error}
\end{center}
\end{figure}

\begin{table}
\caption{Weighted errors~\eqref{eq: E*}~and
convergence rates~\eqref{eq: rate} for different $\alpha$, when
$u_{0h}=I_hu_0$.}
\label{table: errors Ih}
\begin{center}
\renewcommand{\arraystretch}{1.2}
\begin{tabular}{c|cc|cc|cc}
$Q_h$&
\multicolumn{2}{c|}{$\alpha=0.25$}&
\multicolumn{2}{c|}{$\alpha=0.50$}&
\multicolumn{2}{c}{$\alpha=0.75$}\\
\hline
  7& 7.79e-02&       & 7.46e-02&       & 7.27e-02&       \\
  15& 4.04e-02&  0.948& 3.86e-02&  0.950& 3.76e-02&  0.952\\
  31& 2.06e-02&  0.973& 1.97e-02&  0.973& 1.91e-02&  0.974\\
  63& 1.04e-02&  0.987& 9.93e-03&  0.987& 9.65e-03&  0.987
\end{tabular}
\end{center}
\end{table}
Physically, the solution~$u$ must be non-negative, but the
oscillations in the discrete initial data~$P_hu_0$ mean that
$U^n_h(x)$ was negative for some values of~$(x,t_n)$ near the points
of discontinuity $(\pi/4,0)$~and $(3\pi/4,0)$.  It is tempting to
choose as the discrete initial data~$u_{0h}=I_hu_0$, the nodal
interpolant~\eqref{eq: nodal}. In this way, $U^0_h=u_{0h}(x)\ge0$ for
all~$x$. However, since
\[
\iprod{u_{0h}-P_hu_0,\chi}=\iprod{u_{0h}-u_0,\chi}
	\le\|u_{0h}-u_0\|\|\chi\|\quad\text{for all $\chi\in\Sh$,}
\]
by choosing $\chi=u_{0h}-P_hu_0$ we see that
\[
\|u_{0h}-P_hu_0\|\le\|u_{0h}-u_0\|=\sqrt{\tfrac23}\,h
	\quad\text{when $u_{0h}=I_hu_0$.}
\]
Thus, Theorem~\ref{thm: error} now yields an error bound of
order~$h+t^{-3\alpha/4}h^2$ (ignoring the log factor), and
Table~\ref{table: errors Ih} indeed shows only first-order
convergence for this choice of initial data.

At the end of Section~\ref{sec: stability}, we remarked that in our
stability estimate the constant tends to infinity as~$\alpha$
approaches~$1$. Since the finite element method is stable in the
classical case~$\alpha=1$, we suspect that the dependence of the
stability constant on~$\alpha<1$ is an artefact of the method of
proof.  To investigate this question numerically, we computed
$\|u_h(t)\|$ for random initial data, that is, when the value
of~$u_{0h}$ at each node was a random number from a uniform
distribution in~$[0,1]$.  In practice, we did not observe any
deterioration in the stability of the method for~$\alpha$ close
to~$1$.

\begin{figure}
\begin{center}
\includegraphics[scale=0.6]{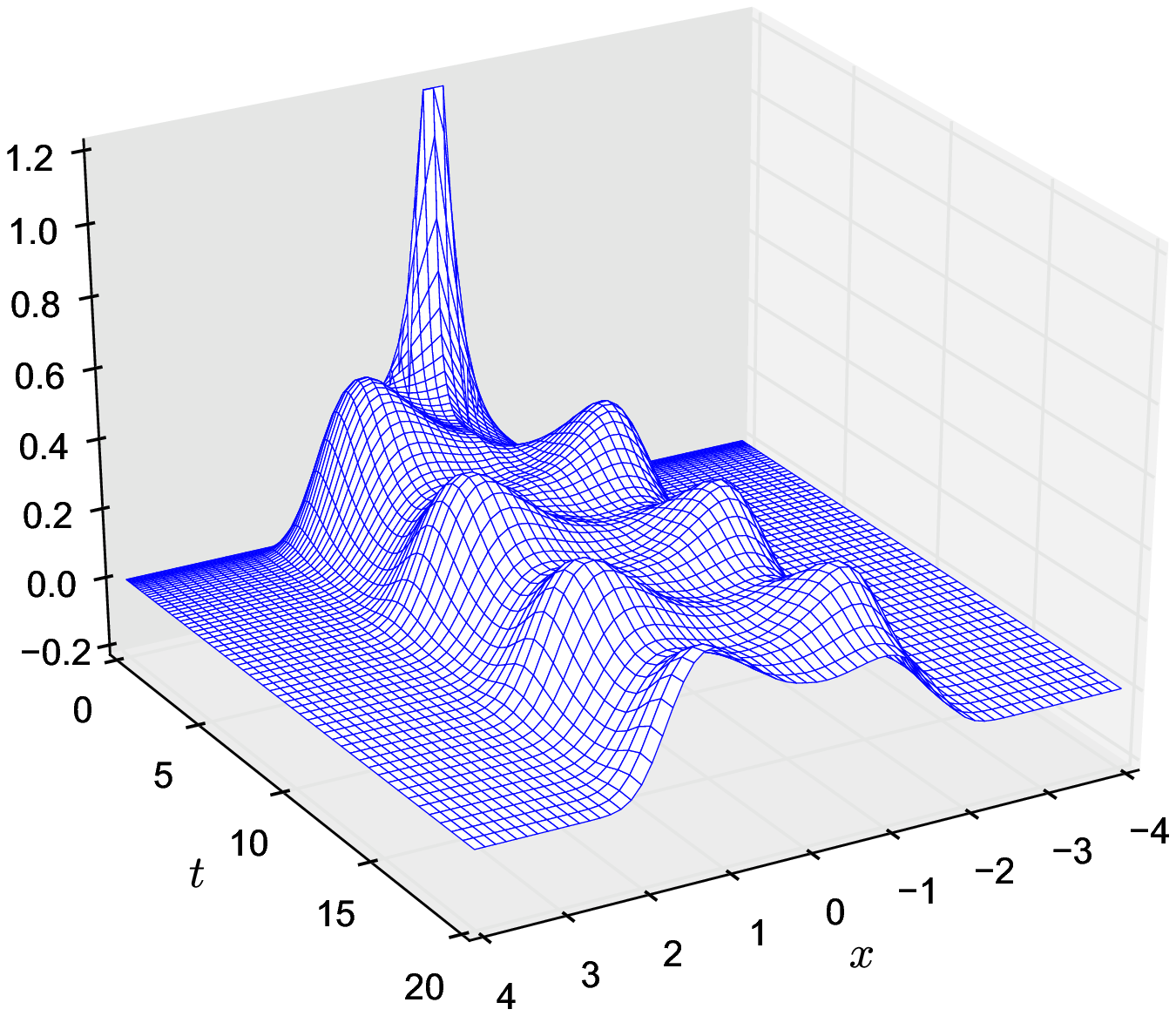}
\caption{Surface plot of a solution using the
potential~\eqref{eq: potential} and imposing the zero-flux boundary
condition~\eqref{eq: Neuman bc};  the part of the surface
where $t<0.005$ is omitted.}
\label{fig: stoch res}
\end{center}
\end{figure}

\begin{figure}
\begin{center}
\includegraphics[scale=0.6]{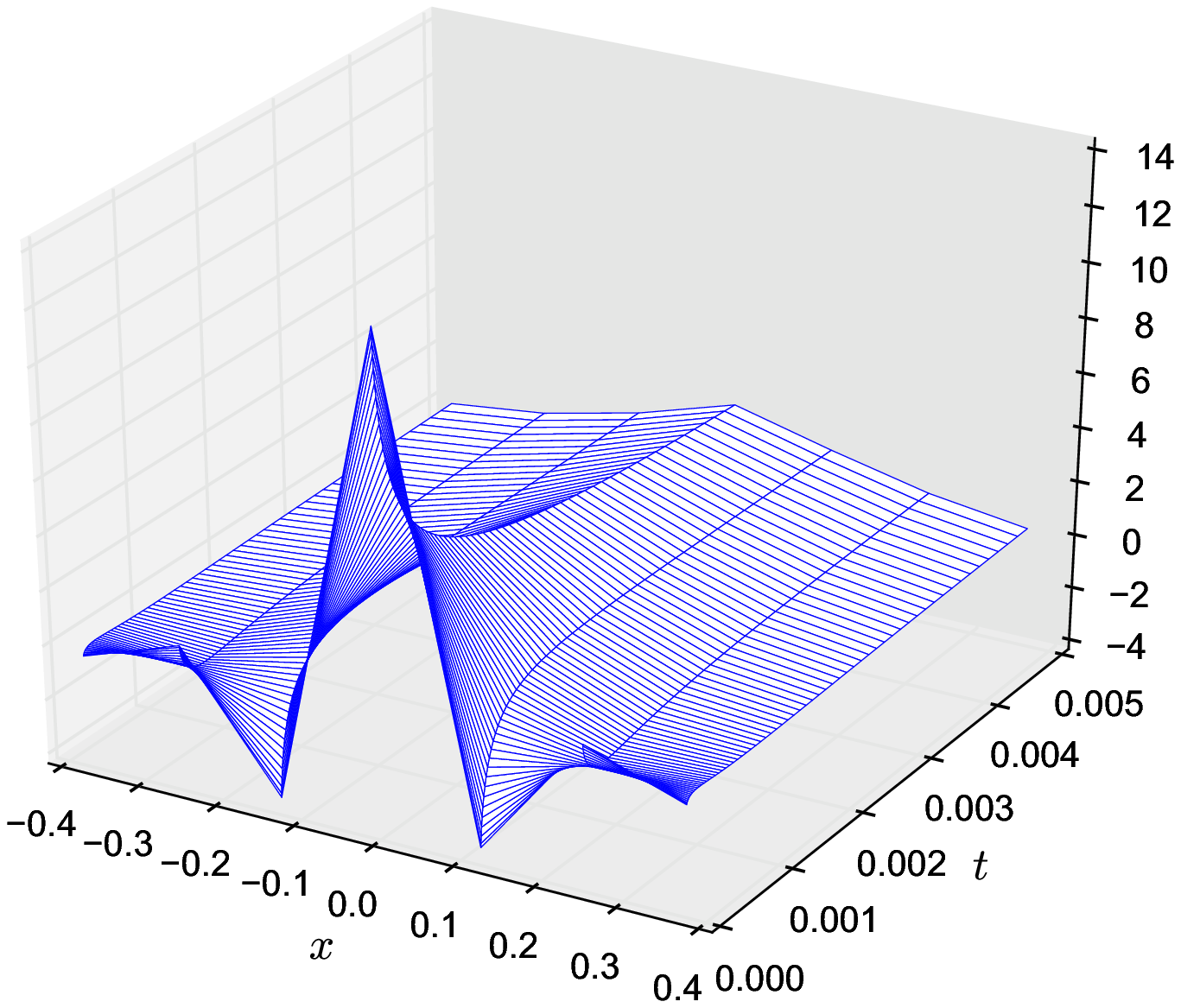}
\caption{Detail of the surface plot showing the spurious
oscillations for~$(x,t)$ near the singularity at~$(0,0)$.}
\label{fig: osc}
\end{center}
\end{figure}

\subsection{Zero-flux boundary condition}
In our second example,
\[
F(x,t)=-\frac{\partial V}{\partial x},\quad
\alpha=0.75,\quad T=20,\quad \Omega=(-L,L),\quad L=4,
\]
where $V$ is a double-well potential perturbed by an oscillation in
time,
\begin{equation}\label{eq: potential}
V(x,t)=\tfrac14x^4-\tfrac12x^2-x\cos t.
\end{equation}
Gammaitoni et al.~\cite{GammaitoniEtAl1998} used this potential for
the classical Fokker--Planck equation ($\alpha=1$) in their study of
stochastic resonance.  We imposed the zero-flux boundary
condition~\eqref{eq: Neuman bc} and chose as the initial
data~$u_0(x)=\delta(x)$.  The solution $u$ then gives the
probability distribution for a single diffusing particle initially
located at~$x=0$.  Since the Dirac delta functional does not belong
to~$L_2(\Omega)$, our stability result (Theorem~\ref{thm: stability})
does not apply, and $P_hu_0$ is not defined. Nevertheless, the
functions in~$\Sh$ are continuous, so by extending the $L_2$ inner
product to a dual pairing we can define the discrete initial
data~$u_{0h}\in\Sh$ by
\[
\iprod{u_{0h},\chi}=\iprod{u_0,\chi}=\iprod{\delta,\chi}=\chi(0)
	\quad\text{for all~$\chi\in\Sh$.}
\]
Figure~\ref{fig: stoch res} shows a surface plot of the numerical
solution using $N=\text{4,096}$ time steps, now with a stronger mesh
grading $\gamma=2$ in~\eqref{eq: graded}, and $Q_h=65$ spatial
degrees of freedom. (Thus the delta function is centred on the
node~$\vec{x}_{33}=0$).  We cut off the initial part of the plot where
$t<0.005$ to avoid the oscillations, shown separately in
Figure~\ref{fig: osc}, which are much larger than was the case for our
first example.  The total mass should be constant and we observed in
practice that $\int_\Omega U^n_h=1$ to ten significant figures,
for~$0\le n\le N$.
%%%%%%%%%%%%%%%%%%%%%%%%%%%%%%%%%%%%%%%%%%%%%%%%%%%%%%%%%%%%%%%%%%%%%%

%%%%%%%%%%%%%%%%%%%%%%%%%%%%%%%%%%%%%%%%%%%%%%%%%%%%%%%%%%%%%%%%%%%%%%
\end{document}